\documentclass[11pt]{article}

\usepackage{epsfig}
\usepackage{graphicx}
\usepackage{amsbsy}
\usepackage{amsmath}
\usepackage{amsfonts}
\usepackage{amssymb}
\usepackage{textcomp}
\usepackage{hyperref}
\usepackage{aliascnt}

\newcommand{\mcm}[3]{\newcommand{#1}[#2]{{\ensuremath{#3}}}} 

\mcm{\tuple}{1}{\langle #1 \rangle}
\mcm{\name}{1}{\ulcorner #1 \urcorner}
\mcm{\Nbb}{0}{\mathbb{N}}
\mcm{\Zbb}{0}{\mathbb{Z}}
\mcm{\Rbb}{0}{\mathbb{R}}
\mcm{\Cbb}{0}{\mathbb{C}}
\mcm{\Qbb}{0}{\mathbb{Q}}
\mcm{\Acal}{0}{\cal A}
\mcm{\Bcal}{0}{\cal B}
\mcm{\Ccal}{0}{\cal C}
\mcm{\Dcal}{0}{\cal D}
\mcm{\Ecal}{0}{\cal E}
\mcm{\Fcal}{0}{\cal F}
\mcm{\Gcal}{0}{\cal G}
\mcm{\Hcal}{0}{\cal H}
\mcm{\Ical}{0}{\cal I}
\mcm{\Jcal}{0}{\cal J}
\mcm{\Kcal}{0}{\cal K}
\mcm{\Lcal}{0}{\cal L}
\mcm{\Mcal}{0}{\cal M}
\mcm{\Ncal}{0}{\cal N}
\mcm{\Ocal}{0}{{\cal O}}
\mcm{\Pcal}{0}{{\cal P}}
\mcm{\Qcal}{0}{{\cal Q}}
\mcm{\Rcal}{0}{{\cal R}}
\mcm{\Scal}{0}{{\cal S}}
\mcm{\Tcal}{0}{{\cal T}}
\mcm{\Ucal}{0}{{\cal U}}
\mcm{\Vcal}{0}{{\cal V}}
\mcm{\Wcal}{0}{{\cal W}}
\mcm{\Xcal}{0}{{\cal X}}
\mcm{\Ycal}{0}{{\cal Y}}
\mcm{\Zcal}{0}{{\cal Z}}
\mcm{\Mfrak}{0}{\mathfrak M}

\mcm{\restric}{0}{\upharpoonright}
\mcm{\upset}{0}{\uparrow}
\mcm{\onto}{0}{\twoheadrightarrow}
\mcm{\smallNbb}{0}{{\small \mathbb{N}}}
\DeclareMathOperator{\preop}{op}
\mcm{\op}{0}{^{\preop}}

{\begin{array}{c}
\setlength{\unitlength}{1em}}%
{\end{array}}

\usepackage{amsthm}

\newcommand{\theoremize}[2]{\newaliascnt{#1}{thm} \newtheorem{#1}[#1]{#2} \aliascntresetthe{#1}}

\newtheorem*{thm*}{Theorem}

\theoremstyle{plain}
\newtheorem{thm}{Theorem}[section]
\theoremize{lem}{Lemma}
\theoremize{skolem}{Skolem}
\theoremize{fact}{Fact}
\theoremize{sublem}{Sublemma}
\theoremize{claim}{Claim}
\theoremize{obs}{Observation}
\theoremize{prop}{Proposition}
\theoremize{cor}{Corollary}
\theoremize{que}{Question}
\theoremize{oque}{Open Question}
\theoremize{con}{Conjecture}

\theoremstyle{definition}
\theoremize{dfn}{Definition}
\theoremize{rem}{Remark}
\theoremize{eg}{Example}
\theoremize{exercise}{Exercise}
\theoremstyle{plain}

\usepackage{verbatim}
\usepackage{enumerate}
\usepackage{cutwin}
\usepackage{listings}
\usepackage[all]{xy}
\usepackage{subfig}
\usepackage{pdfpages}
\usepackage{bm}

\usepackage{subfig}

\title{A characterisation of 3-colourable 3-dimensional triangulations}

\author{Johannes Carmesin, Emily Nevinson, Bethany Saunders
\medskip 
\\
  {University of Birmingham}
}

\usepackage{xcolor}

\definecolor{codegreen}{rgb}{0,0.6,0}
\definecolor{codegray}{rgb}{0.5,0.5,0.5}
\definecolor{codepurple}{rgb}{0.58,0,0.82}
\definecolor{backcolour}{rgb}{0.95,0.95,0.92}

\lstdefinestyle{mystyle}{
    backgroundcolor=\color{backcolour},   
    commentstyle=\color{codegreen},
    keywordstyle=\color{magenta},
    numberstyle=\tiny\color{codegray},
    stringstyle=\color{codepurple},
    basicstyle=\ttfamily\footnotesize,
    breakatwhitespace=false,         
    breaklines=true,                 
    captionpos=b,                    
    keepspaces=true,                 
    numbers=left,                    
    numbersep=5pt,                  
    showspaces=false,                
    showstringspaces=false,
    showtabs=false,                  
    tabsize=2
}

\mcm{\Fbb}{0}{\mathbb{F}}

\usepackage{xr}

\begin{document}

\maketitle
\begin{abstract}
We extend Heawood's theorem on the colourability of plane triangulations to triangulations of 3-space. We prove that a triangulation of 3-space can be edge coloured with three colours if and only if all edges have even degree.
\end{abstract}

\section{Introduction}
In 1898, Heawood proved that a maximal plane triangulation is vertex colourable in three colours if and only if all its vertices have even degrees \cite{heawood1898four}. In this paper we prove a 3-dimensional analogue of this theorem. 

In fact there is quite a natural way to extend theorems about planar graphs to 3-space. Indeed, each 2-dimensional simplicial complex (which we will refer to as a \textit{2-complex} from now) embedded in 3-space has a planar link graph\footnote{The \emph{link graph} at a vertex $v$ of a 2-complex $C$ is the graph $L(v)$ on the edges incident with $v$ in $C$, there is an edge between two vertices in $L(v)$ if they share a face at $v$ in $C$.} at each of its vertices. Hence we are interested in global statements of the simplicial complex that project down to the theorem we are trying to extend in each of its link graphs, see \cite{carmesin2017embedding, carmesin2017embedding2, carmesin2017embedding3, carmesin2017embedding4, carmesin2017embedding5} for details.

A \textit{(proper) edge colouring} of a 2-complex $C$ is a labelling of each of the edges of $C$ such that no two edges that share a face have the same label.
The \textit{(face-)degree} of an edge $e$ in a 2-complex $C$ is the number of faces of $C$ that $e$ is incident with.

\begin{figure}[h]
    \centering
    \includegraphics[width=0.36\textwidth]{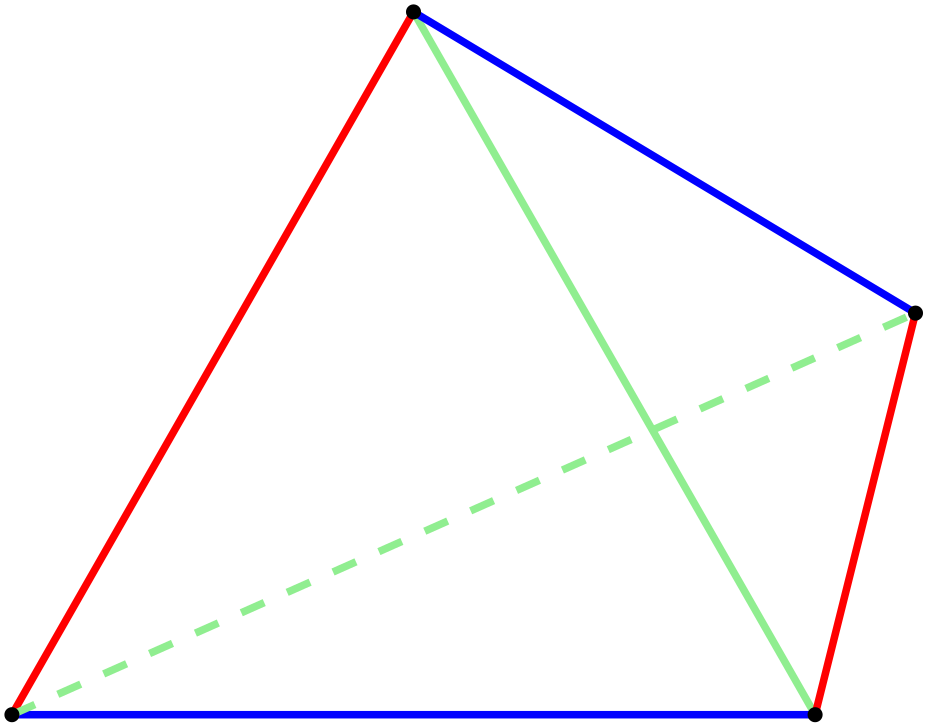}
    \caption{A 3-edge colouring of a tetrahedron. This is an example of a spatial triangulation where all the edges have even degree.}
    \label{fig:3edgecol}
\end{figure}

Intuitively speaking, Heawood's theorem says that local 3-colourings of the faces extend to global 3-colourings of plane triangulations. We extend this \emph{Heawood principle} even further; these 3-colourings of the link graphs can be simultaneously extended to global edge colourings of 2-complexes, as follows: 
\begin{thm*}
    A triangulation of 3-space\footnote{Here, a \emph{triangulation of 3-space} would be a 2-complex where all of the chambers are tetrahedrons.} can be edge coloured with three colours if and only if all edges have even degree.
\end{thm*}

The $n$-dimensional version of this theorem was claimed without proof in the 70s \cite{edwards1977amusing}, however we do not agree that it is as simple of a result as they believed. Indeed, another paper \cite{goodman1978even} claims to prove the 3-dimensional case, however their argument does not seem to work. For further details on this, see the concluding remarks. 

For basics and background, refer to Diestel's book on graph theory \cite{diestel2016graph}, Hatcher's book on Algebraic topology \cite{hatcher2000algebraic}, and the paper series on the 3-dimensional Kuratowski embeddings \cite{carmesin2017embedding, carmesin2017embedding2, carmesin2017embedding3, carmesin2017embedding4, carmesin2017embedding5}.

\section{3-Colourability of Eulerian 2-complexes}\label{3col}
\begin{dfn}
The \textit{Spatial Line Graph} $L(\mathcal{K})$ of a 2-complex $\mathcal{K}$ represents the face adjacencies of the edges of the 2-complex. The vertex set of $L(\mathcal{K})$ consists of the edges of $\mathcal{K}$, if two edges in $\mathcal{K}$ have a face in common, their corresponding vertices in $L(\mathcal{K})$ are adjacent.
\end{dfn}

\begin{figure}[h]
    \centering
    \includegraphics[width=0.9\textwidth]{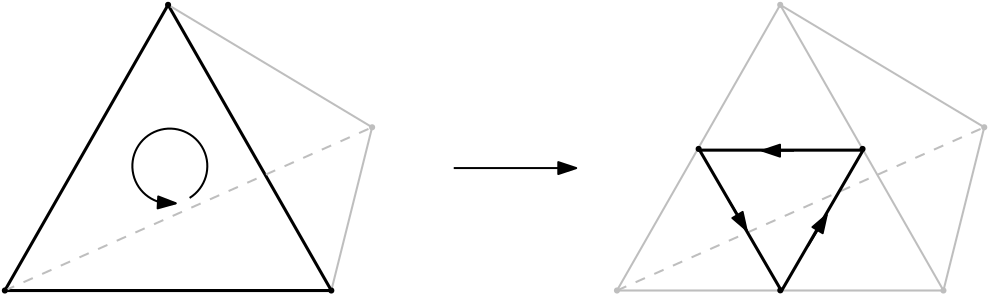}
    \caption{In a directed spacial line graph, for a face with orientation as in the left tetrahedron, we give the edges of the spatial line graph the directions as in the right tetrahedron.}
    \label{fig:directedslg}
\end{figure}

\begin{dfn}
The \textit{Directed Spatial Line Graph}, $DL(\mathcal{K})$, of a 2-complex $\mathcal{K}$ whose faces have an orientation is a directed graph. The directed spacial line graph is the spacial line graph of $\mathcal{K}$ with the added property that the edge joining two vertices in $DL(\mathcal{K})$ is directed with the orientation of the face that the edges are incident to in $\mathcal{K}$, as in \autoref{fig:directedslg}.
\end{dfn}

\begin{dfn}
The \textit{tetrahedron face cycles} in a spacial line graph of a 2-complex are the cycles whose vertices correspond to the three edges that bound a face on a tetrahedron in the 2-complex. 

The \textit{tetrahedron vertex cycles} in a spacial line graph of a 2-complex are the cycles whose vertices correspond to the three edges that are adjacent to a common vertex in a tetrahedron in the 2-complex. 

We refer to the tetrahedron face cycles and the tetrahedron vertex cycles just as the \textit{tetrahedron cycles}.
\end{dfn}

\begin{figure}[h]
    \centering
    \includegraphics[width=0.95\textwidth]{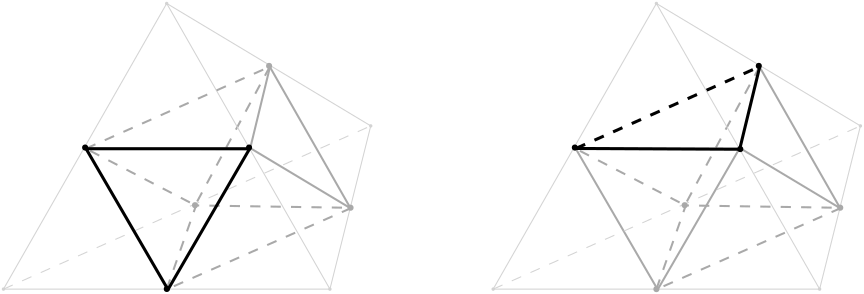}
    \caption{A tetrahedron face cycle (left) and a tetrahedron vertex cycle (right), they are both called tetrahedron cycles.}
    \label{fig:tetcycles}
\end{figure}

\begin{dfn}
An orientation of the faces of a planar graph is \textit{consistent} if each edge has opposite directions in the orientations chosen for each of its two adjacent faces.
\end{dfn}

\begin{lem}
    The incidence vectors of the edges of a 2-complex generate all cycles of the dual matroid\footnote{For the definition of a \emph{dual matroid} see \cite{carmesin2017embedding4}.} over any field, in particular over $\mathbb{F}_2$.
\end{lem}
\begin{proof}
    See \cite{carmesin2017embedding4}.
\end{proof}

\begin{lem}
    A graph is bipartite if and only if all cycles have even length.
\end{lem}
\begin{proof}
    See \cite{diestel2016graph}.
\end{proof}

\begin{lem} \label{bipartite}
    Given a 2-complex embedded in 3-space with all edges of even degree, then the dual graph is bipartite.
\end{lem}
\begin{proof}
    Let $\mathcal{K}$ be a 2-complex embedded in 3-space with all edges having even degree. Consider the dual matroid $M$ of $\mathcal{K}$. For each edge $e$ of $\mathcal{K}$, take the column vector $\boldsymbol{e} \in \mathbb{F}_2^F$ where each face $f \in F$ in $\mathcal{K}$ corresponds to a row of $\boldsymbol{e}$, and there is a 1 in that row if $e$ is incident to $f$, and 0 otherwise. These binary column vectors make up a generating set for the cycle space of the matroid $M$. Each of the vectors $\boldsymbol{e}$ will have even length since each edge in $\mathcal{K}$ has even degree.
    
    Any cycle of $M$ can be written as a sum of the elements in the cycle space. So because we have a generating set for the cycles with all elements even length, this means that any cycle of $M$ will have even length. To see this, consider the sum two binary vectors $\boldsymbol{v}$ and $\boldsymbol{w}$. The length of $\boldsymbol{v}+\boldsymbol{w}$ is just the length of $\boldsymbol{v}$ plus the length of $\boldsymbol{w}$ minus 2 times the number of coordinates where $\boldsymbol{v}$ and $\boldsymbol{w}$ both equal 1.
    
    Now we have that all cycles in $M$ have even length, which means it is bipartite.
\end{proof}

\begin{dfn}
    The \textit{effective length} of a cycle $o$ in a labelled digraph $D$ is the sum over the weights of all edges of $o$ oriented clockwise, minus the sum over the weights of the anti-clockwise edges of $o$.
\end{dfn}

\begin{dfn}
    Given an abelian group $\Gamma$, a $\Gamma$-\emph{co-flow} is an assignment of elements of $\Gamma$ (referred to as weights) to the edges of a directed graph $D$ such that every cycle $o$ of $D$ has effective length 0.
\end{dfn}

\begin{lem} \label{euldig_general}
    Given an abelian group $\Gamma$ and a graph $G$, there is a colouring of the vertices of $G$ with the elements of $\Gamma$ such that adjacent vertices receive different colours if and only if there is a nowhere zero $\Gamma$-co-flow on the edges.
\end{lem}
\begin{proof}
    This follows from Section 6.3 in \cite{diestel2016graph}.
\end{proof}

\begin{cor} \label{euldig}
    A labelled digraph with each cycle effective length $0 \text{ mod } k$, where the label on each edge is $d \in \mathbb{Z}_k\backslash {[0]_k}$, is $k$-vertex-colourable.
\end{cor}
\begin{proof}
    This is a special case of \autoref{euldig_general}.
\end{proof}

\begin{lem}\label{colouringL}
    Let $\mathcal{K}$ be a 2-complex and $L(\mathcal{K})$ the spatial line graph of $\mathcal{K}$. If there exists a $k$-vertex-colouring of $L(\mathcal{K})$ then a $k$-edge-colouring of $\mathcal{K}$ exists. 
\end{lem}
\begin{proof}
    Let $\mathcal{K}$ be a 2-complex with edge set $E$ and set of faces $F$. Let $L(\mathcal{K})$ be the spatial line graph of $\mathcal{K}$, with vertex set $V_{L}$ and edge set $E_{L}$. 
    
    Suppose there exists a $k$-vertex-colouring of $L(\mathcal{K})$, a labelling $c : V_{L} \rightarrow C$ such that $|C|=k$ and for $x,y \in V_{L}$ we have that $c(x) \neq c(y)$ whenever $ \{ x, y \} \in E_{L}$. So no two adjacent vertices in $L(\mathcal{K})$ share the same label. By the definition of a spatial line graph, two vertices in $L(\mathcal{K})$ are adjacent if their corresponding edges in $\mathcal{K}$ are incident to a common face.
    
    We define an edge-colouring of $\mathcal{K}$ as follows: 
    $c':E \rightarrow C'$ where $c'(x)=c(x)$ for $x \in E$. First notice that if $x \in E$ then $x \in V_{L}$ since the vertex set of $L(\mathcal{K})$ is precisely the edge set of $\mathcal{K}$ by definition.
    
    Consider $y, z \in E = V_{L}$ with $ \{ y, z \} \subseteq F'$ for $F' \in F$. So $y$ and $z$ are edges of $\mathcal{K}$ incident to a common face, by the definition their corresponding vertices in $L(\mathcal{K})$ are adjacent. In other words $ \{ y, z \} \in E_{L}$. Hence $c'(y) = c(y) \neq c(z) = c'(z)$ whenever $ \{ y, z \} \subseteq F'$ for some $F' \in F$. So $c'$ is a consistent edge-colouring of $\mathcal{K}$.
    
    Moreover, let $c'(x) \in C'$, then $c'(x) = c(x) \in C$ so we have that $C' \subseteq C$ and $|C'|\leq |C| = k$. Hence a $k$-edge-colouring of $\mathcal{K}$ exists.
\end{proof}

\begin{thm}\label{3colour}
    A triangulation of 3-space can be edge coloured with three colours if and only if all edges have even degree.
\end{thm}
\begin{proof}
        First we show that a triangulation of 3-space that can be edge-coloured with three colours has the property that all edges have even degree.
        
        Let $\mathcal{K}$ be a 2-complex embedded in 3-space with every chamber a tetrahedron and can be coloured using 3 colours. Suppose for contradiction that $\mathcal{K}$ has an edge $e$ of odd degree. 
        
        Consider the link graph at $v$, one of the end vertices of $e$, the vertex that corresponds to $e$ in this graph has odd degree. Since $\mathcal{K}$ is a triangulation of 3-space we must have that all of the link graphs are triangulations of the plane. Then by Heawood's theorem the link graph at $v$ cannot be 3-colourable as it has a vertex of odd degree.
        
        Now we prove the opposite direction, a triangulation of 3-space where all edges have even degree can be edge-coloured with three colours. Let $\mathcal{K}$ be a 2-complex embedded in 3-space with every chamber a tetrahedron and every edge even degree. 
        
        \begin{sublem}
            There exists an orientation of all the faces of $\mathcal{K}$ such that the four faces of each tetrahedron form a consistent orientation of that tetrahedron.
        \end{sublem}
        \begin{proof}
            There are two consistent orientations for the faces in tetrahedrons, the left handed and right handed orientations. 
            
            Because the dual graph is bipartite by \autoref{bipartite}, the tetrahedron in $\mathcal{K}$ can be split into two sets such that if two tetrahedron are in the same set then they do not share a face. We can give all of the tetrahedrons in one of the sets the left handed orientation and give the other set the right handed orientation.
            
            When two tetrahedron meet at a face the two orientations will be mirrors of each other and so the orientation on that face will be the same. So we have an orientation for all of the faces in $\mathcal{K}$ that are on tetrahedrons which form consistent orientations on the tetrahedrons.
        \end{proof}
    
\begin{figure}[h]
    \centering
    \includegraphics[width=0.4\textwidth]{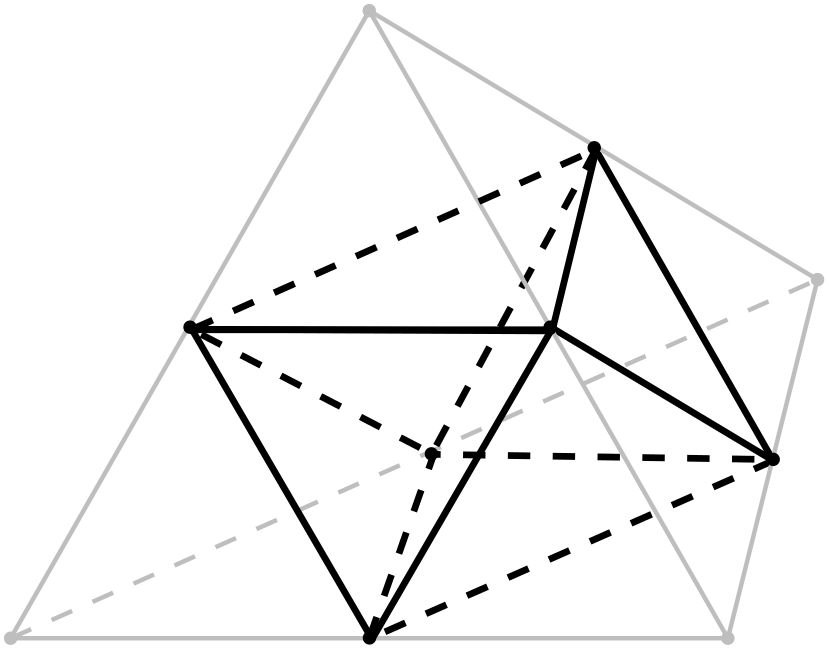}
    \caption{The octahedron (black) in a spacial line graph that is where the tetrahedrons of the original 2-complex (grey) were.}
    \label{fig:slgoctahedron}
\end{figure}
        
        \begin{sublem}
            There exists an orientation of $DL(\mathcal{K})$ such that every cycle has effective length 0 mod 3.
        \end{sublem}
        \begin{proof}
                Give the edges of $DL(\mathcal{K})$ the same orientation that the corresponding face has in $\mathcal{K}$. That is, for a face in $\mathcal{K}$ with edges $x, y, z$ and orientation $xyz$ the corresponding edges in $DL(\mathcal{K})$ will have directions $xy$, $yz$, and $zx$.
                
                It is trivial to see that the tetrahedron cycles of $DL(\mathcal{K})$ all have effective length 0 mod 3. To show that the other cycles in $DL(\mathcal{K})$ have effective length 0 mod 3 we need to build a 2-complex from $DL(\mathcal{K})$ such that all its faces are bounded by cycles of effective length 0 mod 3 and show it is simply connected. For a simply connected 2-complex, the face cycles generate all of the other cycles so our tetrahedron cycles having effective length 0 mod 3 implies that all cycles will have effective length 0 mod 3.
                
                We build the simplicial complex $\mathcal{D}$ from $DL(\mathcal{K})$ by adding a face at every tetrahedron cycle. 
                
                In \cite{carmesin2017embedding2} it was proved that if C is a locally connected simplicial complex embedded in $S^3$, then $C$ being simply connected is equivalent to all the local surfaces being bounded by spheres. So it is enough to show that all local surfaces of $\mathcal{D}$ are bounded by spheres.
                
                In $\mathcal{D}$ we have octahedrons where the tetrahedrons in $DL(\mathcal{K})$ were, see \autoref{fig:slgoctahedron}, which are bounded by tetrahedrons of $DL(\mathcal{K})$ and so these are bounded by spheres. The only other local surfaces are the voids where the vertices of the tetrahedrons in $DL(\mathcal{K})$ were. These are bounded by the new faces we added and so are also bounded by spheres and we are done.
        \end{proof}
        
        Now we have that $DL(\mathcal{K})$ is a digraph with every cycle having effective length 0 mod 3. So by \autoref{euldig} $DL(\mathcal{K})$ is 3-colourable.
        
        Then by \autoref{colouringL} there exists a 3-colouring of $\mathcal{K}$.
\end{proof}

\begin{prop}\label{4col}
    The vertices of a triangulation of 3-space are 4-colourable iff the edges are 3-colourable.
\end{prop}
\begin{proof}
    Take a triangulation of 3-space with a colouring on the vertices using 4 colours. Label the 4 colours with the vectors $\left(\begin{smallmatrix}0\\0\end{smallmatrix}\right), \left(\begin{smallmatrix}0\\1\end{smallmatrix}\right), \left(\begin{smallmatrix}1\\0\end{smallmatrix}\right), \left(\begin{smallmatrix}1\\1\end{smallmatrix}\right)$ from the abelian group $\mathbb{F}_2 \times \mathbb{F}_2$. Then for each edge, the new colour would be the sum of the two colours on the end vertices, over $\mathbb{F}_2 \times \mathbb{F}_2$. This will  give us 3 colours as the vector $\left(\begin{smallmatrix}0\\0\end{smallmatrix}\right)$ cannot appear, if it did it would mean that the two end vertices have the same colour which is not possible under a proper colouring. Also, it is easy to see that if the three vertices on a triangle have different colours, then the edges will also get different colours, so this 3-colouring on the edges is proper.

    Take a triangulation of 3-space with a 3-colouring on the edges. Replace the 3 colours on the edges with the vectors $\left(\begin{smallmatrix}0\\1\end{smallmatrix}\right), \left(\begin{smallmatrix}1\\0\end{smallmatrix}\right), \left(\begin{smallmatrix}1\\1\end{smallmatrix}\right)$. For any triangle we will have exactly one copy of all three vectors present so the sum of the vectors on the triangle cycle is zero. Since the triangulation is simply connected, we have that all of the cycles are generated by the face cycles, i.e. the triangles. So the vectors on every cycle sums to zero. Now we have shown that this is a nowhere zero co-flow, and we can use \autoref{euldig_general} to show that we have a 4-colouring on the vertices.
\end{proof}

We can now restate our theorem as follows.

\begin{cor}
    For a triangulation of 3-space, the vertices are 4-colourable iff all edges have even degree.
\end{cor}
\begin{proof}
    Combine \autoref{4col} and \autoref{3colour}.
\end{proof}

\section{Concluding remarks}
Note that our proof actually proves the stronger statement with the 3-sphere replaced by a general homology sphere.
However the generalisation of our theorem to any 3-manifold is false. To see this, consider the following example.

\begin{eg}
    Take a triangulation of 3-space with all edges even degree. All such triangulations have a unique 3-colouring up to permutation of colours. Take two tetrahedron that do not intersect with any common faces, edges, or vertices. Remove their interiors from the manifold and identify their boundaries in such a way that the colouring isn't compatible. Now we have a triangulation of a 3-manifold where all edges have even degree but isn't 3-colourable.
\end{eg}

In the 70s a paper was published \cite{goodman1978even} that claims to prove \autoref{3colour}. In their proof it is claimed that in a triangulation of a simply connected space every loop of 3-simplices is a sum of simple loops\footnote{A \textit{simple loop} is 3-simplices around an interior edge.}.
However, consider the loop of 3-simplices, as in \autoref{fig:falseproof}, obtained by gluing together 3-simplices in a linear way along faces and then identifying a vertex of the first 3-simplex with a vertex of the last.
This is a loop of 3-simplices that could occur in a such a triangulation, but it is not clear how this is always a sum of simple loops or how one could extend a colouring of a simple loop to a colouring of this loop.

\begin{figure}[h]
    \centering
    \includegraphics[width=0.4\textwidth]{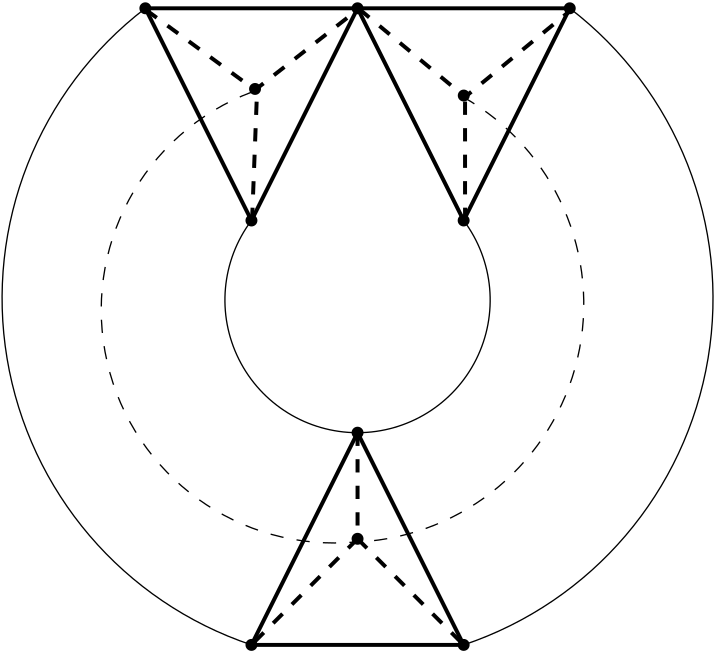}
    \caption{A loop of 3-simplices obtained by gluing together 3-simplices in a linear way along faces and then identifying a vertex of the first 3-simplex with a vertex of the last 3-simplex.}
    \label{fig:falseproof}
\end{figure}

The $n$-dimensional version of \autoref{3colour} is open. The statement is as follows.

\begin{con}
    Let $C$ be a triangulation of $S^d$. Then the 1-skeleton of C is colourable with $d+1$ colours if and only if all its $(d-2)$-faces are incident with an even number of $(d-1)$-faces.
\end{con} 

In a future note \cite{4colour3D} we discuss edge colourings of general 2-complexes and give an upper bound for the chromatic number of embeddable general 2-complexes.

\bibliographystyle{plain}
\bibliography{literatur}

\end{document}